\documentclass{amsart}
\usepackage[usenames]{color}
\usepackage{amsmath, amssymb, latexsym, multicol, color, enumerate,comment,booktabs}
\usepackage{amsthm}
\usepackage{float}
\usepackage{tikz,url}
\usepackage{graphicx, caption}
\usepackage{subcaption}
\usepackage[all]{xy}

\usepackage{algorithm}
\usepackage[noend]{algpseudocode}
\makeatletter
\def\BState{\State\hskip-\ALG@thistlm}
\makeatother

\DeclareMathOperator{\Gal}{\mathrm{Gal}}
\newtheorem{thm}{Theorem}

\newtheorem{lem}[thm]{Lemma}
\newtheorem{claim}[thm]{Claim}
\newtheorem{definition}[thm]{Definition}

\newtheorem{cor}[thm]{Corollary}
\newtheorem{rmk}[thm]{Remark}

\newcommand{\NN}{\mathbb{N}}
\newcommand{\ZZ}{\mathbb{Z}}
\newcommand{\QQ}{\mathbb{Q}}

\newcommand{\cO}{\mathcal{O}}

\newcommand{\A}[2]{\mathcal{H}_{#1}^{(#2)}}

\newcommand{\Extend}[2]{{\mathrm{Extend}}\left(#1,#2\right)}

\newcommand{\Q}{{\mathcal{Q}}}
\newcommand{\cH}{{\mathcal{H}}}
\newcommand{\R}{{\mathfrak{Q}}}

\makeatletter
\newcommand{\addresseshere}{%
  \enddoc@text\let\enddoc@text\relax
}
\makeatother

\title{Imaginary Multiquadratic Fields of Class Number Dividing $2^m$}
\author{Amy Feaver\textsuperscript{1}}
\address{\textsuperscript{1} The King's University, 9125 50 St NW, Edmonton, AB, Canada}
\email{amy.feaver@kingsu.ca}
\author{Anna Pusk\'{a}s\textsuperscript{2}}
\address{\textsuperscript{2} University of Massachusetts, Amherst Lederle Graduate Research Tower, Office 1124, Amherst, MA 01002, United States of America}
\email{puskas@math.umass.edu}


\begin{document}

%--------------------Abstract --------------------------%
\begin{abstract}
This paper gives a method to find all imaginary multiquadratic fields of class number dividing $2^{m},$ provided the list of all imaginary quadratic fields of class number dividing $2^{m+1}$ is known. We give a bound on the degree of such fields. As an application of this algorithm, we compute a complete list of imaginary multiquadratic fields with class number dividing $32.$ 
\end{abstract}
%--------------------Abstract --------------------------%

\maketitle

\section{Introduction}

An \textit{$n$-quadratic number field}, for $n$ a nonnegative integer, is any field $K$ of degree $2^n$ over $\QQ$, which is formed by adjoining the square root of rational integers (radicands) to $\QQ .$ If $n\geq2$, the field is also called a \textit{multiquadratic} number field.  A multiquadratic field is said to be \textit{imaginary} if any of the radicands are negative.
%of $\ell$ rational integers to $\QQ$ for some $\ell\in\NN$. That is, $K=\QQ(\sqrt{a_1},...,\sqrt{a_\ell})$ for $a_1,...,a_\ell\in\ZZ$. 

In this paper we discuss the class number of imaginary multiquadratic fields, particularly in the case where the class number is $2^m$ for a nonnegative integer $m$. 
%We first establish some notation in the next subsection, and then follow with a survey of previous results on this problem.

We begin with a review of previous results. 

\subsection{Previous Results}\label{subsect:prev_review}

One of the first known discussions of class numbers of number fields can be found in section V of Gauss's book \textit{Disquisitiones Arithmeticae} ~\cite{Gau} which was published in 1801. Using the language of quadratic forms, Gauss made several significant conjectures and proved some results about the class numbers of quadratic fields. 

Gauss showed that the class number of $\QQ(\sqrt{-a})$ is equal to 1 for \[a\in\{1,2,3,7,11,19,43,67,163\}.\] He conjectured that his list was complete. This conjecture was proven to be true through a series of results, the most significant contributors being Heegner in 1952 ~\cite{Hee52}, Baker in 1966 ~\cite{Bak66_2} and Stark in 1967 ~\cite{Sta67}. In addition to contributing to this list, Baker and Stark also found a complete list of imaginary quadratic fields of class number 2 between 1966 and 1975 ~\cite{Bak66}, ~\cite{Bak71}, ~\cite{Sta75}. These imaginary quadratic fields of class number 2 are the fields $\QQ(\sqrt{-a})$ for  
\[a\in \{5,6,10,13,15,22,35,37,51,58,91,115,123,187,235,267,403,427\}.\]

Then, in 1992, Arno completed found a complete list of imaginary quadratic fields class number $4$.

\begin{thm} \cite[Theorem 7.]{Arn92} There are $54$ imaginary quadratic of class number $4$; $24$ have ideal class group $(\ZZ/2\ZZ)^2$ and $30$ have ideal class group $(\ZZ/4\ZZ)$. The corresponding values $a$ for these fields $K=\QQ(\sqrt{-a})$ are as follows:

\begin{tabular}{|l|l|}
\hline
Cl($K$) & $a$ \\
\hline
 $(\ZZ/2\ZZ)^2$ & $21, 30, 33, 42, 57, 70, 78, 85, 93, 102, 130, 133, 177, 190, 195, 253, 435,$\\
& $483, 555, 595, 627, 715, 795, 1435$\\
\hline
 $(\ZZ/4\ZZ)$ & $14, 17, 34, 39, 46, 55, 73, 82, 97, 142, 155, 193, 203, 219, 259, 291, 323,$\\
& $355, 667, 723, 763, 955, 1003, 1027, 1227, 1243, 1387, 1411, 1507, 1555$ \\
\hline
\end{tabular}

\end{thm}

Finding lists of imaginary quadratic fields of class number $m$, for a fixed nonnegative integer $m$ is frequently referred to as finding a solution to the Gauss class number $m$ problem. That is, the list of all imaginary quadratic fields of class number 1 is referred to as a solution to the Gauss class number 1 problem. A lot of work has been done on this problem in addition to the lists presented in the theorem above: Oesterl\'{e} solved the class number 3 problem in 1983 ~\cite{Oes83}. Arno completed the classification for class number 4 in 1992 ~\cite{Arn92}, and Wagner completed the class number 5, 6 and 7 classifications in 1996 ~\cite{Wag96}. Then, in 1998, Arno, Robinson and Wheeler solved Gauss' class number $m$ problem for odd values of $m$ with $9\leq m\leq23$ ~\cite{ARW}. 

After completing his classifications, Wagner remarked that the class number 8 case seemed to be impossible. The techniques for solving the Gauss class number $m$ problem became more difficult for integers with a larger number of (not necessarily distinct) prime divisors. Since $8$ has three prime divisors, the problem was thought to be too difficult to solve. However, in 2004, Watkins used computational methods to complete the Gauss class number problem for all imaginary quadratic fields of class number $\leq100$ ~\cite{Wat04}.

It is natural to ask whether or not similar lists exists for higher degree imaginary multiquadratic fields. The first such extension of these results was obtained by Brown and Parry in 1974:

\begin{thm} \cite[Main Theorem]{BP74}
There are exactly 47 imaginary biquadratic fields of class number $1$. These are the fields $\QQ\left(\sqrt{a_1},\sqrt{a_2}\right)$ with the sets $\{a_1,a_2\}$ as follows:

\begin{center}
\begin{tabular}{lllll}

$\{-1,2\}$  & $\{2,-3\}$   & $\{-3,5\}$    & $\{-7,5\}$    & $\{-11,17\}$   \\
$\{-1,3\}$  & $\{2,-11\}$  & $\{-3,-7\}$   & $\{-7,-11\}$  & $\{-11,-19\}$  \\
$\{-1,5\}$  & $\{-2,-3\}$  & $\{-3,-11\}$  & $\{-7,13\}$  & $\{-11,-67\}$  \\
$\{-1,7\}$  & $\{-2,5\}$   & $\{-3,17\}$   & $\{-7,-19\}$  & $\{-11,-163\}$ \\
$\{-1,11\}$ & $\{-2,-7\}$  & $\{-3,-19\}$  & $\{-7,-43\}$  & $\{-19,-67\}$  \\
$\{-1,13\}$ & $\{-2,-11\}$ & $\{-3,41\}$   & $\{-7,61\}$   & $\{-19,-163\}$ \\
$\{-1,19\}$ & $\{-2,-19\}$ & $\{-3,-43\}$  & $\{-7,-163\}$ &$\{-43,-67\}$  \\
$\{-1,37\}$ & $\{-2,29\}$  & $\{-3,-67\}$  &           & $\{-43,-163\}$ \\
$\{-1,43\}$ & $\{-2,-43\}$ & $\{-3,89\}$   &           & $\{-67,-163\}$ \\
$\{-1,67\}$ & $\{-2,-67\}$ & $\{-3,-163\}$ &           &            \\
$\{-1,163\}$ &          &  				 &           &            \\

\end{tabular}
\end{center}
\end{thm}

This result was soon followed (in 1977) by a result of Buell, Williams and Williams on imaginary biquadratic fields of class number 2. We state their theorem below, with some added information about the Hilbert class field and genus field of each of these imaginary biquadratic fields:

\begin{thm}~\cite{BWW77} There are exactly 160 imaginary biquadratic fields $K$ of class number 2; these are given below. Of these fields, exactly 20 have the property that the Hilbert class field $H$ is not the genus field $M$ of $K$.

The 20 fields with $H\neq M$ have radicand lists 

\begin{center}

\begin{tabular}{llll}

$\{-1,17\}$   	& $\{-3,13\}$  	& $\{-7,2\}$    	&  $\{-11,5\}$   	\\
$\{-1,73\}$		& $\{-3,73\}$  	& $\{-7,29\}$   	&  $\{-11,113\}$  	\\
$\{-1,97\}$ 		& $\{-3,97\}$     	& $\{-7,37\}$  	&  $\{-11,137\}$   \\
$\{-1,193\}$   	& $\{-3,241\}$  	& $\{-7,109\}$   	&  $\{-19,17\}$	\\
$\{-2,17\}$		& $\{-3,409\}$ 	&  			&  $\{-19,73\}$	\\
$\{-2,41\}$

\end{tabular}
\end{center}

The remaining 140 fields are $\QQ(\sqrt{a_1},\sqrt{a_2})$ where $a_1$ and $a_2$ are given in the following table:

\begin{center}
\begin{tabular}{|l|l|}
\hline
$a_1$ &  $a_2$ \\
\hline
$-1$ & $6, 10, 15, 21, 22, 33, 35, 57, 58, 91, 93, 115, 133, 177, 253, 403$ \\
\hline
$-2$ & $3, 11, 21, -5, -13, -15, -35, -37, -91, -115, -235, -403, -427$ \\
\hline
$-3$ & $7, 11, 14, 19, 31, 59, 161, 209, 59, -5, -10, -22, -35, -58, -115, -187, -235$ \\
\hline
$-5$ & $2, 7, 23, -7, -67, 47, -43, -163$\\
\hline
$-6$ & $-11, -19, -22, -43, -67, -163$\\
\hline
$-7$ & $3, 6, 19, 69, -10, -13, -15, -51, -115, -123, -187, -235, -267, -403$ \\
\hline
$-10$ & $46, 94, -35, -43, -67, -163$ \\
\hline
$-11$ & $3, 23, 57, -13, -51, -58, -91, -123, -403, -427$ \\
\hline
$-13$ & $7, 31 -67, -163$\\
\hline
$-15$ & $3, 6, 21, 69, 141, -43, -67, -163$\\
\hline
$-19$ & $3, 7, 33, -13, -22, -37, -58, -91, -123, -403$ \\
\hline
$-22$ & $-43, -67, -163$\\
\hline
$-35$ & $-43, -67, -115, -163, -235$ \\
\hline
$-37$ & $-43, -163$\\
\hline
$-43$ & $-58, -115, -235, -267, -427$ \\
\hline
$-51$ & $-163, -187$\\
\hline
$-58$ & $-163$ \\
\hline
$-67$ & $-123, -235, -403$ \\
\hline
$-91$ & $-163, -403$\\
\hline
$-115$ & $-163, -235$ \\
\hline
$-163$ & $-187, -235, -267, -403$ \\
\hline
\end{tabular}
\end{center}
\end{thm}

Both Feaver~\cite{F17} and Yamamura ~\cite{Yam} independently determined the imaginary triquadratic fields of class number 1. Feaver further proved that for $n>3$ there are no imaginary $n$-quadratic fields of class number 1.

\begin{thm} \cite[Theorem 22]{F17}
There are exactly 17 imaginary triquadratic fields of class number $1$. These are the fields $\QQ\left(\sqrt{a_1},\sqrt{a_2},\sqrt{a_3}\right)$ with the sets $\{a_1,a_2,a_3\}$ as follows:

\vspace{1em}

\begin{center}
\begin{tabular}{lllll}
$\{-1,2,3\}$  & $\{-1,3,5\}$  & $\{-1,7,5\}$  & $\{-2,-3,-7\}$  & $\{-3,-7,5\}$  \\
$\{-1,2,5\}$  & $\{-1,3,7\}$  & $\{-1,7,13\}$ & $\{-2,-3,5\}$ & $\{-3,-11,2\}$  \\
$\{-1,2,11\}$ & $\{-1,3,11\}$ & $\{-1,7,19\}$ & $\{-2,-7,5\}$ & $\{-3,-11,-19\}$ \\
           & $\{-1,3,19\}$ &            &            & $\{-3,-11,17\}$ \\
\end{tabular}.
\end{center}

\vspace{1em}

\end{thm}

\subsection{Summary of Results}\label{sect:results}

Building on these results, and following the algorithm described in Sections \ref{sect:biquad_algorithm} and \ref{sect:multiquad_algorithm} we obtain the following. 

\begin{thm}\label{thm:biquad} There are:
\begin{itemize}
\item 408 imaginary biquadratic fields of class number 4,
\item 1186 of class number 8,
\item 2749 of class number 16, and
\item 6657 of class number 32.
\end{itemize}
\end{thm}

\begin{thm}\label{thm:triquad} There are:
\begin{itemize}
\item 27 imaginary triquadratic fields of class number 2,
\item 48 of class number 4,
\item 146 of class number 8,
\item 280 of class number 16, and
\item 484 of class number 32.
\end{itemize}
\end{thm}

\begin{thm}\label{thm:quadriquad} 
There are no imaginary quadriquadratic fields of class number 1. In addition, we have the following. 
\begin{itemize}
\item There is exactly one imaginary quadriquadratic field of class number 2. This field has radicand list $\{-1, 2, 3, 5\}$.
\item There are five imaginary quadriquadratic fields of class number 4. They have radicand lists 
\[\{-2, -3, 5, -7\}, \{-1, 2, 3, 7\}, \{-1, 3, 5, 7\}, \{-1, 2, 5, 7\}, \{-1, 2, 3, 11\}.\]
\item Three of class number 8, with radicand lists given by
\[\{2, -3, 5, -7\}, \{-1, 3, 7, 13\}, \{-1, 3, 5, 11\}.\]
\item Six of class number 16, with radicand lists
\[\{-1, 5, 7, 11\}, \{-1, 3, 5, 13\}, \{-1, 3, 11, 17\}, \{-1, 2, 3, 17\}, \{-1, 2, 7, 11\}, \{-2, -3, 5, -11\}.\]
\item Twelve of class number 32, with radicand lists
\[\{2, -3, -5, 7\}, \{-1, 3, 7, 19\}, \{-2, -3, -5, -7\}, \{-2, -3, -11, 17\}, \{-2, 3, 5, -7\}, \{-1, 6, 7, 10\},\]
 \[\{-2, 3, -5, -7\}, \{-1, 5, 6, 7\}, \{-1, 3, 5, 19\}, \{-1, 3, 7, 10\}, \{2, -3, -5, -11\}, \{-2, -3, -11, -19\}\]
\end{itemize}
\end{thm}

In addition, we conclude that there are no other imaginary multiquadratic fields of class number dividing $32.$

\subsection{Methods and overview}\label{subsect:methods_overview}

This paper describes a computational method to find all imaginary multiquadratic fields whose class number divides $2^m$. As input, the method requires a list of all imaginary quadratic fields whose class number divides $2^{m+1}$. The method is carried out for $m=5$; that is, for class numbers dividing $2^5=32$. This requires a list of all imaginary quadratic fields of class number dividing $2^6=64$. Additionally, we present statistics related to the effectiveness of our computational strategy.

The method is recursive. We select imaginary $n$-quadratic fields with the desired class number from extensions of imaginary $(n-1)$-quadratic fields by appropriately chosen radicands. This relies on a sequence of lemmas presented in Section \ref{sect:multiquad_lemmata}. To make the method computationally feasible, we rely on a class number proved by the first named author in \cite[Theorem 20]{F17}. This formula expresses the class number of an imaginary multiquadratic field in terms of the class numbers of all its imaginary quadratic subfields. 

The structure of the paper is as follows. Notation is introduced and some preliminaries on imaginary multiquadratic fields are reviewed and \cite[Theorem 20]{F17} is recalled (Theorem \ref{classnumberformula}) in Section \ref{sect:notation_and_preliminaries}. Section \ref{sect:boundondegree} gives a bound on the degree over $\QQ$ of the fields we seek. In Section \ref{sect:biquad_algorithm} we describe the first step of the method: the algorithm to find all imaginary biquadratic fields of class number dividing $32.$ After presenting the lemmas necessary for the further steps in Section \ref{sect:multiquad_lemmata}, the rest of the method is described in Section \ref{sect:multiquad_algorithm}. The details of the algorithm are included in the Appendix. 

\subsection{Acknowledgements} We are very thankful to Frank Calegari for suggesting this project, and his extremely helpful advice and guidance. We are also grateful for his feedback regarding the manuscript, including the proof of Lemma \ref{lem:Hilbertcontain}. During the first half of working on this paper, the second author was a postdoctoral fellow at the University of Alberta, supported by Manish Patnaik's Subbarao Professorship in Number Theory and NSERC grant. 

\section{Preliminaries}\label{sect:notation_and_preliminaries}

We begin by introducing some notation and terminology for imaginary multiquadratic fields and their radicand lists. Proofs of the statements below, and more detail can be found in ~\cite{F17}. %\pacomm{Should we also give a more generic reference for basic stuff? Perhaps also for some of the lemmata in the multiquadratic tools section.}

%Let $\ell\in\ZZ$ with $1\leq \ell$, and let $a_1,...,a_\ell$ be squarefree rational integers such that  $a_i\neq0,1.$ The field  $K=\QQ(\sqrt{a_1},...,\sqrt{a_\ell})$ is called an $n$-quadratic field if $[K:\QQ]=2^n;$ multiquadratic if $n\geq 2.$ Note that in this case $1\leq n\leq \ell .$ $K$ is imaginary if at least one of the $a_i$ are negative. 
%\pacomm{omit this or merge it wih very first paragraph of the paper? Depending on how the introduction changes.}

\begin{definition}
Let $K$ be an $n$-quadratic field, and let $\ell $ be an integer such that $1\leq n\leq \ell$. A list of squarefree rational integers $a_1,...,a_\ell$ ($a_i\neq0,1$, $1\leq i\leq \ell$) is called a \em{radicand list} for $K$ if $K=\QQ(\sqrt{a_1},...,\sqrt{a_\ell})$. Further, $\{a_1,...,a_\ell\}$ is called a \em{primitive radicand list} if $n=\ell$.  
\end{definition}

%It is redundant to adjoin more than $n$ square roots to $\QQ$ when representing an $n$-quadratic field. It is also unnecessary to use radicands that are not squarefree.  We want to describe these fields as cleanly as possible, so we define the following:

\begin{definition} Let $n\in\NN$, and let $K$ be an $n$-quadratic field. A radicand list for $K$ that consists of exactly $2^n-1$ distinct, squarefree rational integers is called a {\em{complete radicand list}} for $K$. We denote this list by $\mathcal{Q}_K$ (or, more simply by $\mathcal{Q}$ when the field $K$ is clear from the context).
\end{definition}

For example, the triquadratic field with primitive radicand list $\{-1,2,3\}$ has complete radicand list $\mathcal{Q}=\{-1,-2,2,-3,3,-6,6\}.$ Though any multiquadratic field has several primitive radicand lists, there is always exactly one complete radicand list. In fact, given a primitive radicand list $\{a_1,...,a_n\}$ for an $n$-quadratic field $K$, we may construct a complete radicand list for $K$ in a straightforward manner: %. This list will be: 
\begin{equation}\label{eq:allrads}
\mathcal{Q}_K=\left\{sf \left(\prod_{i\in I}a_i\right)\mid \emptyset\subsetneq I\subseteq\{1,...,n\}\right\},
\end{equation}
where $sf(a)$ denotes the squarefree part of the integer $a$; i.e. $sf(20) = 5$.

We recall two results on radicand lists. 

\begin{lem}\label{numberOfQuadraticSubfields1} \cite[Lemma 8]{F17} Let $K$ be an $n$-quadratic field. The set of squarefree integers $\{a_1,...,a_{2^n-1}\}$ is a complete radicand list for $K$ if and only if the fields given by $\QQ(\sqrt{a_i})$, $1\leq i\leq2^n-1$, are exactly the $2^n-1$ distinct quadratic subfields of $K$.
\end{lem}

When negative radicands are present, we consider real and imaginary subfields:

\begin{lem}\label{numberOfQuadraticSubfields} \cite[Lemma 9]{F17} Let $K$ be an imaginary $n$-quadratic number field with $n>1$. Then $K$ has $2^{n-1}$ imaginary quadratic subfields and $2^{n-1}-1$ real quadratic subfields.
\end{lem}

We will use $\Q_K^-$ (respectively, $\Q_K^+$) to denote the negative (respectively, positive) squarefree radicands in $\Q_K.$ These correspond to the imaginary quadratic (respectively, real quadratic) subfields of $K.$ 

If $\cH$ is any set of $n$-quadratic fields, $n\in\mathbb{Z}^+$, we use $\Q^-(\cH)$ to refer to the set of all negative radicands of all of the fields in $\cH .$ That is, $\Q^-(\cH)=\bigcup_{K\in \cH}\Q_K^-.$

%\pacomm{Insert reference to general theory on class numbers maybe?}

Let $h_K$ denote the class number of the field $K.$ We refer the reader to ~\cite{Mar} if they wish to read the definition and background information on class numbers. For $r\in \Q ,$ $r<0,$ we use the shorthand $h_r:=h_{\QQ(\sqrt{r})}$ for the class number of an imaginary quadratic field.  Theorem \ref{classnumberformula} below relates the class number of the imaginary multiquadratic field $K$ to the class numbers of its imaginary quadratic subfields. We introduce the notation $P(K)$ for the product of the class numbers of all imaginary quadratic subfields of $K:$ 
\begin{equation}\label{eq:def:PofK}
P(K):=\prod _{r\in \mathcal{Q}_K^-} h_r. 
\end{equation}

For the remainder of the paper, we shall use $\A{n}{2^m}$ to refer to the set of imaginary $n$-quadratic fields whose class number divides $2^m.$ The algorithm of this paper relies on the input $\A{1}{2^{m+1}}.$ We shall use $\R_i$ to denote the set of radicands that determine imaginary quadratic fields of class number $2^i;$ i.e. 
$$\R_i=\{r\in \ZZ_{<0}\mid h_r=2^i \}=\Q^{-} (\A{1}{2^{i}}\setminus \A{1}{2^{i-1}}).$$ 
To determine $\A{n}{32}$ (for any $1\geq n$), we rely on the knowledge of the sets $\R_i$ for $0\leq i\leq 6.$

\subsection{A class number formula}

We state a theorem that will be the foundation of the arguments in the rest of the paper. To do so, we need further notation. 

Let $K$ be an imaginary $n$-quadratic field and $k$ a real $(n-2)$-quadratic subfield of $K$. Then $K/k$ is a $V_4$ extension, so there exist three intermediate fields $k_i$, $1\leq i\leq 3$ such that $k \subsetneq k_i\subsetneq K$. We will let $k_1,k_2$ denote the imaginary fields and $k_3$ the real field.

For any number field $L$, let $E(L)$ denote the unit group of $\mathcal{O}_L$. We define the \textit{unit index} $q(K/k)$ by 
\[q(K/k)=[E(K):E(k_1)E(k_2)E(k_3)].\]

\begin{thm}\label{classnumberformula} {\cite[Theorem 20]{F17}} With the notation above assume that there is as least one odd prime ramified in the extension $K/k$. Then we have
\[h_K=\left(\frac12\right)^{2^{n-1}-1}QP h_{k_3}.\]
Here, $P=P(K)$ is the product of the class numbers of all imaginary quadratic subfields of $K$, and $Q\in\ZZ_{>0}$ is a product of some unit indices $q(K'/k')$ with $k'\subseteq K$.
\end{thm}

There are two important things to note in the above theorem. First, for an $n$-quadratic field $K$ with $n\geq3$, it is always possible to choose an $(n-2)$-quadratic subfield $k$ such that an odd prime ramifies in $K/k$. Also, the theorem does not describe exactly which unit indices are the factors of $Q$, as there is no straightforward way of expressing this. However, $q(K/k)$ is always one of these factors.

\section{A bound on the degree}\label{sect:boundondegree}

As we saw in Section \ref{subsect:prev_review}, the list of imaginary $n$-quadratic fields of class number $1$ stops at $n=3,$ i.e. with triquadratic fields. We shall see that if the class number is $2$, then $n$ is at most $4$, and that there is only one imaginary $4$-quadratic field of class number $2$. In this section, we show that if an imaginary $n$-quadratic field has class number $2^m,$ $m\geq 0$, then one may give a bound on $n$ that increases at a very slow rate when compared to the growth of $m$.

%\pacomm{Can we say something more general here? I mean asymptotically, with a similar technique? Does this give a useful bound for much higher $n$s?}
%\afcomm{Haven't gotten to looking into this yet. Sorry! I'll try to do this before we put the paper on the ArXiV... but it might not happen on time (end of the semester...)}

\begin{thm}\label{degbound} Let $K$ be an imaginary $n$-quadratic field with class number $2^m$, $m\geq0$. We have the following relationship between $m$ and $n$:
\begin{enumerate}
\item If $n=6$, $m\geq 7$.
\item If $n=7$, $m\geq 37$.
\item If $n=8$, $m\geq 99$.
\item If $n\geq 9$ then $m\geq 2^{n-1}-34$.
\end{enumerate}
\end{thm}

\begin{rmk} Specifically, the above theorem tells us that if a $6$-quadratic field has a class number that is a power of $2,$ then the class number must be at least $128$. For $7$-quadratic fields, a class number that is a power of $2$ must be substantially larger: at least $2^{37}=137,438,953,472.$ %$2^{35}=34,359,738,368.$
\end{rmk}

Before proceeding with the proof of Theorem \ref{degbound} we state a lemma. 

\begin{lem}\label{primeradicands} If $K$ is an imaginary $n$-quadratic field, then at most $n$ of its imaginary quadratic subfields have radicands that are equal to a negative prime number or $-1$.
\end{lem}
\begin{proof} % The proof is by contradiction. % 
Let $\mathcal{Q}$ be the complete radicand list for $K$ and $\{-a_1,-a_2,...,-a_k\}$ the subset of $\mathcal{Q}$ such that $a_i$ is either equal to $1$ or a rational prime. It follows from \eqref{eq:allrads} that $\QQ\subsetneq \QQ(\sqrt{-a_1}) \subsetneq \QQ(\sqrt{-a_1},\sqrt{-a_2}) \subsetneq \cdots \subsetneq \QQ (\sqrt{-a_1},\sqrt{-a_2},...,\sqrt{-a_k}).$ Since $[K:\QQ]=2^n,$ this implies $k\leq n.$
\end{proof}

Now we are ready to prove Theorem \ref{degbound}.
\begin{proof} Recall that there are 9 imaginary quadratic fields of class number 1. The radicands of these fields are $\R_0=\{-1,-2,-3,-7,-11,-19,-43,-67,-163\}$. Observe that this list consists of  $-1$ and negative primes. Note that of the $18$ imaginary quadratic fields of class number $2,$ three have prime radicands: $-5, -13$ and $-37;$ two have radicands that are products of elements of $\R_0$: $-6=(-1)\cdot(-2)\cdot(-3)$ and $22=(-1)\cdot(-2)\cdot(-11)$. The remaining radicands are the negative product of two primes. The proof relies on these observations. 

Let $K$ be an imaginary $n$-quadratic field. Recall that $K$ has $\Q _K^-=2^{n-1}$ distinct imaginary quadratic subfields. Let $h_K=2^m$ for some nonnegative integer $m$. Theorem \ref{classnumberformula} states:
\begin{equation*}
2^m=\left(\frac12\right)^{2^{n-1}-1}QPh_{k_3}\geq \left(\frac12\right)^{2^{n-1}-1}Ph_{k_3}.
\end{equation*}
Since $h_K$ is a power of $2$, all of the class numbers contributing the product $P$ must be powers of $2$ as well. Note furthermore that if $n\geq 6,$ we have $2\mid h_{k_3}$ ~\cite[Corollary 15]{F17} . 

We consider the case where $n\geq 9$ and $6\leq n<9$ separately. First let $n\geq 9$. Then we have that 
\begin{align*}
2^m & \geq 2^{1-2^{n-1}}\cdot P\cdot h_{k_3} \\
&\geq 2^{1-2^{n-1}}\cdot 1^9\cdot 2^{18}4^{2^{n-1}-27}\cdot2 \\
&= 2^{1-2^{n-1}}\cdot 1^9 \cdot 2^{18}2^{2^{n-1}+2^{n-1}-27-27+1} \\
%&= \left(\frac12\right)^{-1} 2^{-35}2^{2^{n-1}} \\
&=  2^{-34}2^{2^{n-1}}.
\end{align*}

%\begin{align*}
%2^m & \geq \left(\frac12\right)^{2^{n-1}-1}P\cdot h_{k_3} \\
%&\geq \left(\frac12\right)^{2^{n-1}-1}1^92^{18}4^{2^{n-1}-27}\cdot2 \\
%&= \left(\frac12\right)^{2^{n-1}-1}1^9 2^{18}2^{2^{n-1}+2^{n-1}-27-27+1} \\
%&= \left(\frac12\right)^{-1} 2^{-35}2^{2^{n-1}} \\
%&=  2^{-34}2^{2^{n-1}}.
%\end{align*}

Thus we conclude
\begin{equation*}
m\geq 2^{n-1}-34\text{ whenever }n\geq9.
\end{equation*}

Now assume that $6\leq n<9$. Then by Lemma \ref{primeradicands}, we have that 
$$|\Q_K^-\cap (\R_0\cup \{-5,-13,-37\})|\leq n.$$ 
(Recall that $\{-5,-13,-37\}$ are the prime elements of $\R_1.$) Note that $6\leq n $ implies that $2^{n-1}-n>15=18-3=|\R_1|-3.$
%\begin{equation*}
%e:={\min\{18-3, 2^{n-1}-n\}}.
%\end{equation*}
Then
\begin{align*}
2^m & \geq 2^{1-2^{n-1}}\cdot P\cdot h_{k_3} \\
& \geq 2^{2-2^{n-1}}\cdot 1^n\cdot 2^{15}\cdot 4^{2^{n-1}-n-15}\\
& =2^{2-2^{n-1}}\cdot 2^{15} 2^{2\cdot (2^{n-1}-n-15)}\\
& =2^{2^{n-1}-2n-13}.
\end{align*}

Thus, we have the following:
\begin{itemize}
\item if $n=8$, then $m\geq 2^7-16-13=99$;
\item if $n=7$, then $m\geq 2^6-14-13=37$;
\item if $n=6$, then $m\geq 2^5-12-15=7$.
\end{itemize}
\end{proof}

\section{Imaginary Biquadratic Fields of Class Number Dividing 32}\label{sect:biquad_algorithm}

The results of the last section suggest that finding all imaginary multiquadratic fields of class number dividing 32 is tractable, since we know we will not have to push our search beyond imaginary 5-quadratic fields. In fact, we find that our bound on $n$ is not tight, and we can terminate our search sooner.

Our technique for finding $n$-quadratic fields whose class number divides $2^m$ is different for $n=2$ and $n\geq3$. We discuss the case $n=2$ in this section. We make use of the following theorem, due to Lemmermeyer:

\begin{thm}\label{thm:KurodaForImBiquad}
\cite[Theorem 1]{Lem94} 
Let $K$ be an imaginary biquadratic number field with class number denoted by $h_K$. Denote by $k_1$ and $k_2$ the imaginary quadratic subfields of $K$, and by $k_3$ the real quadratic subfield of $K$. Let $h_{k_1},h_{k_2},h_{k_3}$ be the class numbers of $k_1,k_2,k_3$, respectively. For $L$ any number field, let $E(L)$ denote the unit group of $\mathcal{O}_L.$

Then we have 
\begin{equation}\label{eq:KurodaForImBiquad}
h_K=\frac{1}{2} h_{k_1} h_{k_2} h_{k_3}\cdot [E(K):E(k_1)E(k_2)E(k_3)].
\end{equation}

\end{thm}

Recall that we use the notation $\A{2}{c}$ to denote the set of all imaginary biquadratic fields of class number dividing $c$. Theorem \ref{thm:KurodaForImBiquad} has the following obvious corollary: 

\begin{cor}\label{Cor:Thm20Cor}
If $K\in \A{2}{2^m}$ then $h_{k_1}h_{k_2}h_{k_3}\mid 2^{m+1}.$
%\begin{equation}\label{eq:Thm20Cor}
%h_{k_1}h_{k_2}h_{k_3}\mid 2^{m+1}.
%\end{equation}
\end{cor}

We seek to obtain $\A{2}{32}$. Using the notation of Corollary \ref{Cor:Thm20Cor} we have  $h_{k_1} h_{k_2} h_{k_3}\mid 64$ for any $K\in \A{2}{32}.$ 

To determine $K\in \A{2}{32}$ we first find a list of all biquadratic fields such that $h_{k_1} h_{k_2}\mid 64$, as described in Algorithm \ref{biquad_candidates_1} below. This yields $82,531$ biquadratic candidate fields. 

Since computing the class number of a biquadratic field tends to be computationally several times more expensive than computing the class number of a quadratic field, we find $h_{k_3}$ and only keep those candidate fields with $h_{k_1} h_{k_2} h_{k_3}\mid 64$. The process for doing this is described in Algorithm \ref{biquad_candidates_2}. After running this second algorithm, we have $11,607$ candidate fields remaining.

\begin{algorithm}[h!]
\caption{Biquadratic fields with $h_{k_1}h_{k_2}\leq64$.}\label{biquad_candidates_1}

\begin{flushleft}
\textbf{Precondition:} We have sets of radicands, $\R_i$, for $0\leq i\leq 6$. Each set contains all of the squarefree radicands that are the radicands of imaginary quadratic fields of class number $2^i$ as in section \ref{sect:notation_and_preliminaries}. \\
\textbf{Postcondition:}  The variable \texttt{biquad\_candidates} is a set of ordered triples. In each triple, the first two integers are squarefree negative integers $a$ and $b$ such that $P':=h_a h_b \leq 64$. The third integer in this ordered triple is the product of class numbers $P'$.
\end{flushleft}

\begin{algorithmic}[1]
\State \texttt{biquad\_candidates}$\gets\{\}$
\For{$i\gets0$ \textbf{to} 6}
\For{$j\gets0$ \textbf{to} 6}
\If{$i+j\leq6\And i< j$}
\State \texttt{rad\_lists\_with\_prod} $\gets \{(a,b,2^{i+j}):a\in \R_i,b\in \R_j\}$
\State \texttt{biquad\_candidates} $\gets$\texttt{biquad\_candidates}$\cup$\texttt{rad\_lists\_with\_prod}
\EndIf
\If{$i+j\leq6\And i= j$}
\State \texttt{rad\_lists\_with\_prod}$\gets \{(a,b,2^{2i}):a\in \R_i,b\in \R_i, a<b\}$
\State \texttt{biquad\_candidates}$\gets$\texttt{biquad\_candidates}$\cup$\texttt{rad\_lists\_with\_prod}
\EndIf
\EndFor
\EndFor
\end{algorithmic}
\end{algorithm}

%
%.
%
%.

%The algorithm above yields $82,531$ biquadratic candidate fields, so we run the below algorithm to reduce this list further. 

%%Gotta edits this one!

\begin{algorithm}[h!]
\caption{Biquadratic fields with $h_{k_1}h_{k_2}h_{k_3}\leq64$.}\label{biquad_candidates_2}

\begin{flushleft}
\textbf{Precondition:} We have created the set \texttt{biquad\_candidates} produced in Algorithm \ref{biquad_candidates_1} above. \\
\textbf{Postcondition:} We have a set of radicand lists of length 2, consisting of two negative integers $a$ and $b$. These all have the property that $h_a h_b h_{ab} \leq 64$.
\end{flushleft}

\begin{algorithmic}[1]
\State \texttt{biquad\_rad\_lists}$\gets\{\}$
\For{$tuple$ \textbf{in} \texttt{biquad\_candidates}}
\State $a\gets$ first element in the tuple
\State $b\gets$ second element in the tuple
\State $P'\gets$ third element in the tuple
\State $c\gets sf(ab)$
\If{$P'h_c\leq64$ }
\State \texttt{biquad\_rad\_lists}=\texttt{biquad\_rad\_lists}$\cup\{\{a,b\}\}$
\EndIf
\EndFor
\end{algorithmic}
\end{algorithm}

We compute the class number of each of these $11,607$ candidate fields individually using Sage. What we find is that $11,207$ of these candidate fields are in $\A{2}{32}$, and the $400$ remaining fields have class number equal to $64$. More specifically, we get Theorem \ref{thm:biquad} as a result. Note that the number of biquadratic fields of class number $1$ and $2$ were previously known; our computations reproduced the list of previously known fields.

\section{Tools for Fields with $n\geq3$}\label{sect:multiquad_lemmata}

In order to find all imaginary $n$-quadratic fields with $n\geq 3$ whose class number divides $32$, we make use of several lemmas. For a number field $K$, we denote by $H_K$ its Hilbert class field. Recall that this is the smallest unramified abelian extension of $K$ and it has the property that $[H_K:K]=h_k$.

\begin{lem}\label{lem:Hilbertcontain}
Let $L/K$ be an extension of number fields. Then $H_K.L\subset H_L.$ %\pacomm{Do we need Galois here?}
\end{lem}

\begin{proof}
Since $H_K/K$ is an Abelian extension, so is $H_K.L/L.$ Similarly, $H_K/K$ is unramified. Then for any $w$ prime in $\cO_L,$ let $v\in \cO_K$ be a prime so that $w$ lies over $v.$ Since $H_K/K$ is unramified at $v,$ $H_K.L/L$ is unramified at $w.$ So $H_K.L$ is an unramified and Abelian extension of $L.$ 
\end{proof}

%\begin{proof} Since $H_K/K$ and $L/K$ are Abelian extensions, so are $H_K.L/L$ and $H_K.L/K$. Let $\mathfrak{p}$ be a prime in $\cO_K$, and $\mathfrak{q}$ be a prime lying over $\mathfrak{p}$ in $\cO_L$. Let $E_{\mathfrak{p}}$ and $E_{\mathfrak{q}}$ denote the inertia fields of the primes $\mathfrak{p}$ and $\mathfrak{q}$ inside the extensions $H_K.L/K$ and $H_K.L/L$, respectively. Clearly $H_K\subseteq E_\mathfrak{p}$ and $L\subseteq E_\mathfrak{q}.$ Thus $H_K.L\subseteq E_\mathfrak{p}.E_\mathfrak{q}=E_\mathfrak{q}\subseteq H_L$.
%\end{proof}

\begin{cor}\label{cor:abelian_divis}
Let $L/K$ be Abelian. Then $[H_L:K]=[H_L:L]\cdot [L:K].$
\end{cor}
%\begin{proof}
%This is immediate from the fact that $K\subseteq L\subseteq H_L.$
%\end{proof}

\begin{lem}\label{lem:primedegextclassnum}
Let $L/K$ be a Galois extension where $[L:K]=p$ is a prime. Then one of the following two holds:
\begin{enumerate}[(i)]
\item If $L/K$ is ramified at some prime ideal of $\cO_K$, then $h_K\mid h_L.$ 
\item If $L/K$ is unramified, then $p\mid h_K\mid p\cdot h_L.$
\end{enumerate}
\end{lem}

\begin{proof}
Note that $\Gal(L/K)\cong \ZZ/p\ZZ$ is Abelian. 

If $L/K$ is ramified, then $K\subseteq H_K\cap L\subsetneq L,$ hence $H_K\cap L=K.$ This implies  $[H_K.L:L] = [H_K:K]=h_K.$ Then by Lemma \ref{lem:Hilbertcontain} %and Corollary \ref{cor:abelian_divis} 
$$h_K=[H_K.L:L]\mid [H_L:L]=h_L.$$

If $L/K$ is unramified, then $L\subseteq H_K.$ This implies $K\subseteq L \subseteq H_K\subseteq H_L$ hence $$h_K=[H_K:K]=[H_K:L]\cdot [L:K]\mid [H_L:L]\cdot [L:K]=h_L\cdot p.$$
Note that $K\subseteq L \subseteq H_K$ also implies $p=[L:K]\mid [H_K:K]=h_K.$ So we have that $p\mid h_K\mid h_L\cdot p .$
\end{proof}

\begin{cor}
Let $L$ be an imaginary multiquadratic field. Suppose that $K\subset L$ has index 2. If $L/K$ is ramified at any prime $p$, then $h_K\mid h_L$.
\end{cor}
%\begin{proof}
%Since $L$ is multiquadratic, $L/\QQ$ is Galois and hence so is $L/K.$ Since $[L:K]=2$ is a prime, this then follows from Lemma \ref{lem:primedegextclassnum}.
%\end{proof}

\begin{claim}\label{claim:conjugationfixed_imag_equiv}
Suppose that $L$ is an imaginary multiquadratic field, $\mathfrak{p}$ a prime, $\sigma \in \Gal(L/\QQ)$ such that $\sigma$ fixes $\mathfrak{p}$ ($\sigma \in I_\mathfrak{p}$) and $K=L^\sigma.$ Then $K$ is imaginary if and only if $\sigma $ is not complex conjugation. 
\end{claim}
\begin{proof}
$L$ is imaginary, hence complex conjugation is in $\Gal(L/\QQ)$. Let us denote it by $\gamma .$ Assume $\sigma \neq \gamma .$ Then $\langle \gamma, \sigma \rangle \gneq \langle \sigma \rangle .$ Since $\Gal(L/\QQ)$ is Abelian, both are normal subgroups, hence $K^{\gamma }=L^{\langle \gamma, \sigma \rangle }\subsetneq L^ \sigma =K$ that is, $K$ is imaginary.
If $\sigma =\gamma $ and hence $K=L^{\gamma }$ then $K$ is not imaginary. 
\end{proof}

\begin{lem}\label{lem:atleastdeg8_composite} Suppose that $L$ is an imaginary multiquadratic field with $[L:\QQ]\geq 8$. Then $L$ is the compositum of subfields $K\subset L$ of index two with the following properties:
\begin{enumerate}
\item $L/K$ is ramified at at least one prime,
\item $K$ is imaginary.
\item $h_K$ divides $h_L$.
\end{enumerate}
\end{lem}

\begin{proof}
Now suppose $L$ is ramified at $\mathfrak{p}$, and let $\sigma\in\Gal(L/\QQ)$ lie in $I_\mathfrak{p}$. Then $K=L^\sigma$ (the fixed field of $\sigma$) is imaginary if and only if $\sigma$ is not equal to complex conjugation. Suppose that there exist non-trivial elements $\sigma\in I_p$ and $\tau\in I_q$ that are distinct from each other and are not complex conjugation (it might be the case that $\mathfrak{p}=\mathfrak{q}=(2)$). Then $L/L^\sigma$ and $L/L^\tau$ are both ramified at at least one prime and $L^\sigma$ and $L^\tau$ are imaginary. Moreover, the compositum of $L^\sigma$ and $L^\tau$ is $L$. If there do not exist such a pair of primes and elements, then the union of $I_\mathfrak{p}$ for all $\mathfrak{p}\mid\Delta_L$ can only contain at most two non-trivial elements, and hence generates a subgroup of $\Gal(L/\QQ)$ of order at most four. Yet the quotient by this group would then be unramified everywhere and of degree at least $[L:\QQ]/4$, which is impossible if $[L:\QQ]\geq8$.
\end{proof}

Recall the notation $\mathcal{Q}(\mathcal{H})^-$ and $\A{n}{c}$ from section \ref{sect:notation_and_preliminaries}. For a set of fields $\mathcal{H}$ and a set of radicands $\mathcal{Q},$ let $\Extend{\mathcal{H}}{\mathcal{Q}}$ denote the set of all fields $K(\sqrt{r})$ such that $K\in \mathcal{H}$ and $r\in \mathcal{Q}$. 
The following lemma is a corollary of Lemma \ref{lem:atleastdeg8_composite}. %It will be of use in computing the set  $\A{n}{2^{m}}$ where $n\geq 3.$ 

\begin{lem}\label{lem:highermultiquadstep}
Let $m$ and $n$ be positive integers, $n\geq 3.$ Keeping the above notation, we have 
%let $\mathcal{I}$ denote the set of imaginary quadratic subfields with radicands in $\mathcal{Q}\left(\A{n-1}{2^{m}}\right)^-$. Then
\begin{equation}\label{eq:highermultiquadstep}
\A{n}{2^{m}}\subseteq \Extend{\A{n-1}{2^{m}}}{\mathcal{Q}^-\left(\A{n-1}{2^{m}}\right)\cap \mathcal{Q}^-\left(\A{1}{2^{m+1}}\right)},
\end{equation}
i.e. every imaginary $n$-quadratic field with class number dividing $2^m$ can be written as an extension $K(\sqrt{r})$ where:
\begin{enumerate}
\item $K$ is an imaginary $(n-1)$-quadratic field with class number $h_K|2^m,$ and 
\item $\QQ(\sqrt{r})$ occurs as a subfield of an imaginary $(n-1)$-quadratic field with class number dividing $2^m,$ and itself has class number $h_r|2^{m+1}.$
\end{enumerate}
\end{lem}

\begin{proof}
The proof is by induction on $n,$ and relies heavily on Lemma \ref{lem:atleastdeg8_composite}. 

Let $n\geq 3$ and $L\in \A{n}{2^{m}}$. Then $[L:\QQ]=2^n\geq 8,$ hence by Lemma \ref{lem:atleastdeg8_composite} $L$ is a compositum of imaginary subfields of index two, whose class number divides the class number of $L.$ Thus there are imaginary multiquadratic fields $K_1$ and $L_2$ such that $L=K_1.L_2,$ $[L:K_1]=[L:L_2]=2,$ and $h_{K_1},\ h_{L_2}$ divide $h_L.$ This implies that $K_1,\ L_2\in \A{n-1}{2^{m}}.$ 

Now if $n-1\geq 3,$ we can repeat the same argument with $L_2$ in place of $L,$ and find that $L_2=K_2.L_3,$ where $K_2,\ L_3\in \A{n-2}{2^{m}}.$ Since $[L:K_1]=2$ and $K_1.L_2=L$ we may assume that $L_3\not\subseteq K_1,$ and hence we have $L=K_1.L_3,$ where $K_1\in \A{n-1}{2^{m}}$ and $L_3\in \A{n-2}{2^{m}}.$ 

Repeating this argument a few times, we find imaginary multiquadratic fields $L_2,L_3,\ldots ,L_{n-1}$ such that $L_i\in \A{n+1-i}{2^{m}}$ and $L=K_1.L_{i}.$ For any $i\leq n-2$ we have that $[L_i:\QQ]=2^{n+1-i}\geq 2^3=8,$ hence the conditions of Lemma \ref{lem:atleastdeg8_composite} apply. 

Thus we have $L=K_1.L_{n-1}$ where $K_1\in \A{n-1}{2^{m}}$ and $L_{n-1}\in \A{2}{2^m}.$ Then by Theorem \ref{thm:KurodaForImBiquad} we can write $L_{n-1}$ as a compositum of two imaginary quadratic fields with class number dividing $2^{m+1}:$ $L_{n-1}=K_{n-1}.L_n$ where $K_{n-1},\ L_n\in \A{1}{2^{m+1}}.$ We may assume that $L_n\not\subseteq K_1.$ Let $L_n=\QQ(\sqrt{r}).$ Note that $L_n\subset L_2\in \A{n-1}{2^{m}}.$ Thus we have $L=K_1.L_n=K_1(\sqrt{r}),$ where $K_1\in \A{n-1}{2^{m}}$ and %$L_n=\QQ(\sqrt{r})$ for some 
$r\in \mathcal{Q}^-\left(\A{n-1}{2^{m}}\right)\cap \mathcal{Q}^-\left(\A{1}{2^{m+1}}\right).$ 
\end{proof}

\section{Imaginary Multiquadratic Fields with $n\geq3$}\label{sect:multiquad_algorithm}

Using the tools developed in the last section, the sets $\R_k$ for $k\leq 6$ and the list of fields $\A{2}{32}$ found in Section \ref{sect:biquad_algorithm}, we compute all imaginary multiquadratic fields of class number dividing 32. %Actual analysis of the algorithms used is forthcoming.

In Section \ref{sect:biquad_algorithm} we used the fact that the class number $h_K$ of a multiquadratic field $K$ is divisible by the product $P(K)$ of the class numbers of all its imaginary quadratic subfields. We once again make use of this statement, concluding the following corollary of Theorem \ref{classnumberformula}:
\begin{cor}\label{cor:Psmall}
Let $K$ be an imaginary $n$-quadratic field whose class number divides $32.$ Then $P(K)|2^{2^{n-1}+4}.$
%\begin{equation}\label{eq:Psmall}
%P(K)|2^{2^{n-1}+4}.
%\end{equation}
\end{cor}

The algorithm to find all imaginary multiquadratic fields of class number dividing $32$ is based on Lemma \ref{lem:highermultiquadstep} and Corollary \ref{cor:Psmall}. The process is recursive by $n.$ The sets $\A{1}{32}$ and $\A{2}{32}$ are known from above. By Lemma \ref{lem:highermultiquadstep}, we look for a list of candidates for $\A{n}{32}$ in 
$$\Extend{\A{n-1}{2^{m}}}{\mathcal{Q}^-\left(\A{n-1}{2^{m}}\right)\cap \mathcal{Q}^-\left(\A{1}{2^{m+1}}\right)},$$
the compositions of fields in $\A{n-1}{32}$ with certain imaginary quadratic fields. Finding the class number of a candidate is computationally expensive. Hence we use Corollary \ref{cor:Psmall} to do a preliminary vetting of candidates before their class numbers are computed. %Furthermore, we write $\R_k=\mathcal{Q}(\A{1}{2^k}\setminus \A{1}{2^k})^-$ for the set of (negative) radicands $r$ such that $\QQ(\sqrt{r})$ is an imaginary quadratic field with class number $2^k.$ Note that we have the set $\R_k$ for $k\leq 6.$

We first describe the method of vetting a single candidate in \ref{subsect:single_candidate}. Section \ref{subsect:_mq_algorithm} gives an overview of how this fits into the process of determining $\A{n}{32}$ for every $n$ recursively. The explicit algorithms are included in Appendix \ref{sect:algorithms}. Section \ref{subsect:statistics} presents some statistics about the process. 

\subsection{A single candidate}\label{subsect:single_candidate}

We wish to estimate the class number of $K(\sqrt{r})$ where $K\in \A{n-1}{32}$ and $r\in \Q^-(\A{1}{64})\cap \Q^-(\A{n-1}{32})$ is a squarefree, negative rational integer. By Lemma \ref{lem:highermultiquadstep} the set $\A{n}{32}$ consists of fields this form. 

Using the notation above, recall that $K$ is an imaginary $(n-1)$-quadratic field with class number dividing $32,$ while $r$ is a negative squarefree integer such that  $h_r\mid 64,$ and $\QQ(\sqrt{r})\subset L$ for a field $L\in \A{n-1}{32}.$ 

First note that $K\subsetneq K(\sqrt{r})$ if an only if $r\notin \mathcal{Q}_K^-.$ Assume this is the case. Then $r<0$ implies that we have 
\begin{equation}\label{eq:candallnegrads}
\Q_{K(\sqrt{r})}^-=\Q_K^-\cup \{r\} \cup \bigcup_{a\in \Q_{K}^+} \{sf(ra)\}.
\end{equation}
Let us write $S=\bigcup_{a\in \Q_{K}^+} \{sf(ra)\},$ then
\begin{equation}
P(K(\sqrt{r}))=P(K)\cdot h_{r}\cdot \prod_{s\in S}h_{s}.
\end{equation}

By Corollary \ref{cor:Psmall}, we have that $K(\sqrt{r})\notin \A{n}{32}$ unless 
\begin{equation}\label{eq:Ppart_small}
P(K)\cdot h_{r}\cdot \prod_{s\in S} h_{s}|2^{2^{n-1}+4}.
\end{equation}

Recall that the set $\R_i$ is known for $0\leq i\leq 6.$ Since $r\in \Q\left(\A{1}{64}\right)^-,$ $h_r$ is known. Write furthermore $S_i=S\cap \R_i$ for $0\leq i\leq 6$ and $S'=S\setminus \bigcup _{i=0}^6 S_i.$ Then $a\in S'$ implies that $h_a$ is either not a power of $2,$ or at least $2^7.$ We have that if
\begin{equation}\label{eq:P_estimate}
P_{\text{estimate}}:=P(K)\cdot h_{r}\cdot \left(\prod _{i=1}^6 2^{i\cdot |S_i|}\right),
\end{equation}
then
\begin{equation}\label{eq:P_from_estimate}
P(K(\sqrt{r}))=P_{\text{estimate}}\cdot \prod_{a\in S'} h_{a}.
\end{equation}

It follows that $K(\sqrt{r})\notin \A{n}{32}$ unless 
\begin{equation}\label{eq:estimate}
P_{\text{estimate}}\cdot 2^{7\cdot |S'|} \mid 2^{2^{n-1}+4}.
\end{equation}

%In the algorithm, a field $K$ will be stored with the following information: 
%\begin{itemize}
%\item a primitive radicand list
%\item the list of negative radicands $\Q_K^-$
%\item the list of positive radicands $\Q_K^+$
%\item the integer $P(K).$
%\end{itemize}
%\pacomm{Note: We should be keeping track of the list of positive radicands of fields. The current code is not quite in this form; I will fix this in the final cleanup - Anna.}

The explicit algorithm for vetting candidates $K(\sqrt{r})$ from the set $$\Extend{\A{n-1}{2^{m}}}{\mathcal{Q}^-\left(\A{n-1}{2^{m}}\right)\cap \mathcal{Q}^-\left(\A{1}{2^{m+1}}\right)}$$ 
is Algorithm \ref{algorithm:single_multiquad_candidate} in Appendix \ref{sect:algorithms}. We note that the algorithm takes $\Q_K^-$ and $\Q_K^+$ as part of the input. One may compute $\Q_{K(\sqrt{r})}^-$ by \eqref{eq:candallnegrads} and similarly, $\Q_{K(\sqrt{r})}^+$ by: 
\begin{equation}\label{eq:candallposrads}
\mathcal{Q}_{K(\sqrt{r})}^+=\mathcal{Q}_K^+\cup \bigcup_{a\in \mathcal{Q}_{K}^-} \{sf(ra)\}.
\end{equation}

The fact that a candidate passes the vetting process in Algorithm \ref{algorithm:single_multiquad_candidate} does not mean that its class number divides $32.$ (The vetting does not consider the factors $Q$ and $h_3$ in Theorem \ref{classnumberformula}.) However, this vetting reduces the number of candidates significantly. Computing the class numbers that appear in $S'$ for one of the candidates is feasible, whereas computing the class numbers of imaginary quadratic subfields of all candidates before vetting is not. (See section \ref{subsect:statistics} for explicit statistics.)

\subsection{The process for $n\geq 3$}\label{subsect:_mq_algorithm}

We are now ready to describe the method of finding all imaginary $n$-quadratic fields with class number dividing $32.$ This happens recursively on $n$. A single segment of the process determines $\A{n+1}{32}$ from $\A{n}{32}.$ The explicit algorithm of one segment is Algorithm \ref{algorithm:multiquad_segment_n} in Appendix \ref{sect:algorithms}. We summarize the process here. 

Since computations make use of $\Q _K^-,$ $\Q _K^+,$ and $P(K),$ it is useful to store a field $K$ in the form of the tuple $K=(\Q _K^-,\Q _K^+,P(K)).$ 

The first segment (the first iteration of Algorithm \ref{algorithm:multiquad_segment_n}, for $n=3$) takes as input the set of fields $\A{2}{32}.$ This was found according to Section \ref{sect:biquad_algorithm} above. Note that for each biquadratic field $K=\QQ(\sqrt{r_1},\sqrt{r_2})$ we have $\Q _K^-=\{r_1,r_2\}$ and $\Q _K^+=\{sf(r_1r_2)\}.$ 

For each following segment, we begin with the set $\A{n-1}{32}$ of imaginary $(n-1)$-quadratic fields whose class number divides $32.$ 

We first compute $\Q^-\left(\A{n-1}{32}\right)\cap \Q^-\left(\A{1}{64}\right).$ 

By Lemma \ref{lem:highermultiquadstep}, all elements of $\A{n}{32}$ are of the form $K(\sqrt{r})$ where $K\in \A{n-1}{32}$ and $r\in \Q^-\left(\A{n-1}{32}\right).$ To avoid having to compute the class number for every such $K(\sqrt{r})$ we first vet every candidate $K(\sqrt{r})$ using Algorithm \ref{algorithm:single_multiquad_candidate}, i.e. the process described above in \ref{subsect:single_candidate}. This produces two sets, ${\mathtt{vet\_n\_P}}$ of fields $L$ where $P(L)$ has been determined, and a set ${\mathtt{vet\_n\_part\_P}}$ of tuples $(\Q_{L}^-,\Q_{L}^+,S'(L),P_{\text{estimate}}(L))$ of candidates $L$ where $P(L)=P_{\text{estimate}}(L)\cdot \prod_{s\in S'(L)}h_s,$ but the class numbers $h_s$ are not known for any $s\in S'(L).$ 

We collect the radicands with ``missing'' class numbers: ${\mathfrak{S}}_n=\bigcup_{L}S'(L),$ where the union is over fields $L$ represented in the set ${\mathtt{vet\_n\_part\_P}}.$ After determining the class number $h_s$ for every $s\in {\mathfrak{S}}_n,$ we can compute $P(L)$ for every $L$ in ${\mathtt{vet\_n\_part\_P}},$ and check if $L$ satisfies the condition of Corollary \ref{cor:Psmall}. If it does, then the tuple $L=(\Q_{L}^-,\Q_{L}^+,P(L))$ is added to the set ${\mathtt{vet\_n\_P}}.$ 

Finally, we compute the class number of every field on the list ${\mathtt{vet\_n\_P}}$ and select the fields that are in $\A{n}{32}.$

The set ${\mathtt{vet\_n\_P}}$ turns out to be a lot smaller than the initial set of $$|\A{n-1}{32}|\cdot |\Q^-(\A{n-1}{32})\cap \Q^-(\A{1}{64})|$$ candidates. Observe that to perform the vetting process and find ${\mathtt{vet\_n\_P}}$, given $\R_i$ ($i\leq 6$) we only needed to compute the class numbers $h_s$ of imaginary quadratic fields with radicands $s\in {\mathfrak{S}}_n.$ To illustrate the advantage of this approach, we note some statistics about sizes of the relevant sets in each iteration of the algorithm. 

\subsection{Statistics}\label{subsect:statistics}

There are $11207$ imaginary biquadratic fields whose class number divides $32,$ i.e. $|\A{2}{32}|=11207.$ We have $\Q ^-(\A{2}{32})=\bigcup_{k=0}^6\R_k,$ a total of $1485$ radicands. 

Composing the $11207$ elements of $\A{2}{32}$ with the $1485$ radicands would be very large, but after Step \ref{step:subroutine} (Algorithm \ref{algorithm:single_multiquad_candidate}) we find that there are only $|{\mathtt{vet\_3\_P}}|=6537$ triquadratic candidates with $P$ completely computed, and $|{\mathtt{vet\_3\_part\_P}}|=495$ triquadratic candidates with $P$ partially computed. These produce a total of $|{\mathfrak{S}}_3|=440$ negative radicands $s$ such that $h_s$ is unknown. After performing Step  \ref{step:remainingP} for $n=3,$ only $13$ of the $495$ candidates from the list ${\mathtt{vet\_3\_part\_P}}$ are added to ${\mathtt{vet\_3\_P}}.$ Of the $6550$ vetted candidates, in Step \ref{step:computeclnum} we find $17, 27, 48, 146, 280, 484$ fields with class number $2^i,$ $i=0,1,2,3,4,5,$ respectively. These together give $|\A{3}{32}|=1002.$

Next, to determine $\A{4}{32},$ we have $|\Q(\A{3}{32})^-|=251$ radicands in Step \ref{step:collectradstoextend}. Out of all the possible extensions of $1002$ triquadratic fields with $251$ new radicands, we find $|{\mathtt{vet\_4\_P}}|=102$ vetted candidates $L$ with $P(L)$ known and a single candidate $L'$ with an estimate of $P(L')$ computed, and a single radicand $(-110)$ whose class number is missing from the computation of $P(L').$ It turns out that $h_{-110}=12,$ and computing the class number of the remaining $102$ candidates reveals that the number of quadriquadratic fields with class number $2^k$ with $k=0,\ldots ,5$ is $0, 1, 5, 3, 6, 12,$ respectively. Hence we have $|\A{4}{32}|=27.$ 

These quadriquadratic fields have $|\Q(\A{4}{32})^-|=48$ negative radicands between them. None of the candidates built from the $27$ fields and $48$ radicands pass the vetting in Algorithm \ref{algorithm:single_multiquad_candidate}, hence we find $\A{5}{32}=\emptyset .$

\addresseshere

\newpage 

\appendix

\section{Algorithms for multiquadratic fields with $n\geq 3$}\label{sect:algorithms}

\begin{algorithm}[h!]
\caption{Vetting of candidate $K(\sqrt{r})$}\label{algorithm:single_multiquad_candidate}
\begin{flushleft}
{\bf{Precondition:}} We have a field $K,$ a radicand $r,$ radicand lists $\R_i$ for $0\leq i\leq 6,$ as well as sets ${\mathtt{vet\_n\_P}}$ and ${\mathtt{vet\_n\_part\_P}}$ such that
\begin{itemize}
\item $K$ is a field in $\A{n-1}{32},$ given by the tuple $K=(\Q_K^-,\Q_K^+, P(K));$
\item $r\in \Q^-(\A{1}{64})^-\cap \Q^-(\A{n-1}{32})$
\item the sets $\R_i$ for $0\leq i\leq 6$ are radicands of fields in $\A{1}{64}$ as in Section \ref{sect:notation_and_preliminaries}. 
\end{itemize}

{\bf{Postcondition:}} There are three possibilities:
\begin{enumerate}[(1)]
\item the field $K(\sqrt{r})$ is eliminated from consideration
\item the field $K(\sqrt{r})$ is not eliminated, and 
\begin{enumerate}[(a)]
\item $P(K(\sqrt{r}))$ is computed, and $K(\sqrt{r})$ is added to the set of vetted candidates with $P$ determined: ${\mathtt{vet\_n\_P}}$. Here $K(\sqrt{r})$ is recorded as a tuple $K(\sqrt{r})=(\Q_{K(\sqrt{r})}^-,\Q_{K(\sqrt{r})}^+,P(K(\sqrt{r}))).$
\item $P_\text{estimate}$ is computed so that $P(K(\sqrt{r}))=P_{\text{estimate}}\cdot \prod_{a\in S'} h_{a}.$ In this case $K(\sqrt{r})$ is added to the set ${\mathtt{vet\_n\_part\_P}}$ of candidates with $P$ partially computed. Elements of this set are tuples $(\Q_{K(\sqrt{r})}^-,\Q_{K(\sqrt{r})}^+,S',P_\text{estimate}).$
\end{enumerate}
\end{enumerate}
\end{flushleft}
\begin{algorithmic}[1]
\If{$h_r\mid 2^{2^{n-1}+4}\cdot P(K)^{-1}\And r\notin \Q_K^-$ }
\State $S\gets \bigcup_{a\in \Q_{K}^+} \{sf(ra)\}$ 
\State $S_i\gets S\cap \R_i$ for $0\leq i\leq 6$ 
\State $S'\gets S\setminus \bigcup _{i=0}^6 S_i$ 
\State $P_{\text{estimate}}\gets P(K)\cdot h_{r}\cdot \left(\prod _{i=1}^6 2^{i\cdot |S_i|}\right)$ \Comment{\eqref{eq:P_estimate}}
\If{$\log_2(P_{\text{estimate}})+7\cdot |S'| \leq 2^{n-1}+4$} \Comment{\eqref{eq:estimate}}
\If{$S'=\emptyset $} 
\State $P(K(\sqrt{r}))\gets P_{\text{estimate}},$
\State $\Q_{K(\sqrt{r})}^-\gets \Q_K^-\cup \{r\} \cup S$ \Comment{\eqref{eq:candallnegrads}}
\State $\Q_{K(\sqrt{r})}^+\gets \Q_K^+\cup \bigcup _{a\in \Q_{K}^-} \{sf(ra)\}$
\State $K(\sqrt{r})=(\Q_{K(\sqrt{r})}^-,\Q_{K(\sqrt{r})}^+,P(K(\sqrt{r})))$
\State ${\mathtt{vet\_n\_P}}\gets {\mathtt{vet\_n\_P}}\cup \{K(\sqrt{r})\}$ 

\Else \Comment{$S'\neq \emptyset $} 
\State $\Q_{K(\sqrt{r})}^-\gets \Q_K^-\cup \{r\} \cup S$ \Comment{\eqref{eq:candallnegrads}}
\State $\Q_{K(\sqrt{r})}^+\gets \Q_K^+\cup \bigcup _{a\in \Q_{K}^-} \{sf(ra)\}$
\State $S'(K(\sqrt{r}))\gets S'$
\State $P_{\text{estimate}}(K(\sqrt{r}))\gets P_{\text{estimate}}$
\State $K(\sqrt{r})=(\Q_{K(\sqrt{r})}^-,\Q_{K(\sqrt{r})}^+,S'(K(\sqrt{r})),P_{\text{estimate}}(K(\sqrt{r}))) $
\State ${\mathtt{vet\_n\_part\_P}}\gets {\mathtt{vet\_n\_part\_P}}\cup \{K(\sqrt{r})\}$
\EndIf
\EndIf
\EndIf
\end{algorithmic}
\end{algorithm}

\begin{algorithm}[h!]
\caption{Determining $\A{n}{32}$ from $\A{n-1}{32}$}\label{algorithm:multiquad_segment_n}
\begin{flushleft}
{\bf{Precondition:}} The set $\A{n-1}{32}$ with each field $K$ given as a tuple $(\Q_K^-, \Q_K^+, P(K));$ the sets $\R_i$ for $0\leq i\leq 6.$

{\bf{Postcondition:}} A set of fields ${\mathtt{n\_quad\_cands}}$ where each field $L$ in the set is given by a tuple $(\Q_L^-, \Q_L^+, P(L));$ and $L$ is an $n$-quadratic field satisfying the condidtions of Corollary \ref{cor:Psmall}.
\end{flushleft}
\begin{algorithmic}[1]
\item ${\mathtt{radicands}}\gets \{\}$
\For{$K\in \A{n-1}{32}$} \label{step:collectradstoextend}
\State ${\mathtt{radicands}}\gets {\mathtt{radicands}}\cup \Q_K^-$ \Comment{${\mathtt{radicands}}=\Q^-(\A{n-1}{32})$}
\EndFor
\For{$i\gets0$ \textbf{to} 6}
\State $\R_i^{(n)}\gets \R_i\cap \Q^-(\A{n-1}{32})$
\EndFor
\State ${\mathtt{vet\_n\_P}}\gets \{\}$
\State ${\mathtt{vet\_n\_part\_P}}\gets \{\}$
\For{$K\in \A{n-1}{32}$}
\State ${\mathtt{bound}}\gets 2^{n-1}+4-\log_2(P(K))$
\State ${\mathtt{radicands\_for\_K}}\gets \bigcup _{i=1}^{{\mathtt{bound}}}\R_i^{(n)}$
\For{$r\in {\mathtt{radicands\_for\_K}}$} \label{step:subroutine}
\State Perform Algorithm \ref{algorithm:single_multiquad_candidate} with $K$ and $r$ 
\EndFor
\EndFor
\State {${\mathfrak{S}}_n\gets \{\}$}
\For{$L\in {\mathtt{vet\_n\_part\_P}}$}
\State {${\mathfrak{S}}_n\gets {\mathfrak{S}}_n\cup S'(L)$}
\Comment{${\mathfrak{S}}_n$ is the set of radicands $s$ with unknown $h_s$}
\EndFor
\State ${\mathtt{missing\_q\_clnums}}\gets\{\}$
\For{$s\in {\mathfrak{S}}_n$} \Comment{Compute the missing class numbers $h_s$}
\State ${\mathtt{missing\_q\_clnums}}\gets {\mathtt{missing\_q\_clnums}} \cup \{(s,h_s)\}$
\EndFor
\For{$L\in {\mathtt{vet\_n\_part\_P}}$}
\State ${\mathtt{new\_P\_estimate}}\gets P_{\text{estimate}}(L)$
\For{$s\in S'(L)$} \Comment{Use ${\mathtt{missing\_q\_clnums}}$ to compute $P(L)$} \label{step:remainingP}
\State ${\mathtt{new\_P\_estimate}}\gets {\mathtt{new\_P\_estimate}}\cdot h_s$
\EndFor
\State $P(L)\gets {\mathtt{new\_P\_estimate}}$
\If{$\log_2(P(L))\leq 2^{n-1}+4$}
\State ${\mathtt{vet\_n\_P}}\gets {\mathtt{vet\_n\_P}}\cup (\Q_{L}^-,\Q_{L}^+,P(L))$ 
\EndIf
\EndFor
\State $\A{n}{32}=\{\}$
\For{$K\in {\mathtt{vet\_n\_P}}$} \Comment{Compute the class number of the surviving candidates.}
\If{$h_K|32$}
\State $\A{n}{32}\gets \A{n}{32}\cup \{K\}$ \label{step:computeclnum}
\EndIf
\EndFor
\end{algorithmic}
\end{algorithm}


\begin{thebibliography}{}

\bibitem{Arn92} Arno, S., {\em The imaginary quadratic fields of class number 4}, Acta Arith. \textbf{60} (1992) 321—334.

\bibitem{ARW} Arno, S., Robinson, M.L., Wheeler, F.S., {\em Imaginary quadratic fields with small odd class number}, Acta Arith. \textbf{83} (1998) 295—330.

\bibitem{Bak66} Baker, A., {\em A remark on the class number of quadratic fields}, Bull. London Math. \textbf{94} (1966) 98—102.

\bibitem{Bak66_2} Baker, A., {\em Linear forms in the logarithms of algebraic numbers}, Mathematika \textbf{13} (1966) 204—216.

\bibitem{Bak71} Baker, A., {\em Imaginary quadratic fields of class number 2}, Ann. of Math. \textbf{94} (1971) 139—152.
  
\bibitem{BP74} Brown, E., Parry, C., {\em The imaginary bicyclic biquadratic fields with class-number 1}, J. Reine Angew \textbf{266} (1974) 118--120.

\bibitem{BWW77} Buell, D.A., Williams, H.C., Williams, K.S., {\em On the imaginary bicyclic biquadratic fields with class-number 2}, Mathematics of Computation \textbf{31} (1977) 1034--1042.

\bibitem{F17} Feaver, A., {\em Imaginary multiquadratic fields of class number 1}, J. Number Theory \textbf{174} (2017) 93--117.

\bibitem{Gau} Gauss, C. F., {\em Disquisitiones Arithmeticae} (1801).

\bibitem{Hee52} Heegner, K., {\em Diophantische Analysis und Modulfunktionen}, Mathematische Zeitschrift \textbf{56} (1952) 227--253.

\bibitem{Lem94} Lemmermeyer, F., {\em Kuroda's class number formula}, Acta Arith. \textbf{66} (1994) 245--260.

\bibitem{Mar} Marcus, D., {\em Number Fields}, Springer (1977).

\bibitem{Oes83} Oesterl\'{e}, J., {\em Nombres de classes des corps quadratiques imaginaires}, S\'{e}m. Bourbaki (1983, 1984).

\bibitem{Sta67} Stark, H.M., {\em A complete determination of the complex quadratic fields of class-number 1}, Michigan Math J. \textbf{14} (1967) 1--27.

\bibitem{Sta75} Stark, H.M., {\em On complex quadratic fields with class number two}, Mathematics of Computation \textbf{29} (1975) 289--302.

\bibitem{Wag96} Wagner, C., {\em  Class number 5,6 and 7}, Mathematics of Computation \textbf{65} (1996) 785--800.

\bibitem{Wat04} Watkins, M., {\em  Class numbers of imaginary quadratic fields}, Mathematics of Computation \textbf{73} (2004) 907 -- 938.

\bibitem{Yam} Yamamura, K., {\em  The determination of the imaginary abelian number fields with class number one}, Mathematics of Computation \textbf{206} (1994) 899--921.

\end{thebibliography}
\end{document}